\documentclass{amsart}
\usepackage{graphicx}
\usepackage[bookmarksnumbered, colorlinks, plainpages]{hyperref}
\usepackage{color}
 \newtheorem{thm}{Theorem}[section]
 \newtheorem{cor}[thm]{Corollary}
 \newtheorem{lem}[thm]{Lemma}
 \newtheorem{prop}[thm]{Proposition}
 \theoremstyle{definition}
 
 \theoremstyle{remark}
 \newtheorem{rem}[thm]{Remark}
 \newtheorem{ex}{Example}
 \numberwithin{equation}{section}
\newcommand{\be}{\begin{equation}}
\newcommand{\ee}{\end{equation}}

\def\zR{\mathbb R}

\def\zN{\mathbb N}

\begin{document}

%
%
%
%
%
%
%
%
%

\title[Asymptotic formulae and inequalities ]
 {Asymptotic formulae and inequalities for point spectrum in max algebra}

%
%
%
%
%
%
%

\maketitle

\indent \indent \indent \indent \authors{S. M. Manjegani, A. Peperko$^*$, H. Shokooh Saljooghi}

\bigskip

\begin{abstract}
We prove  new  explicit asymptotic formulae between  (geometric) eigenvalues in max-algebra and  classical distinguished eigenvalues of nonnegative matrices, which are useful tools  for transferring results between both settings. We establish new inequalities for both  types of eigenvalues of Hadamard products and Hadamard weighted geometric means of nonnegative matrices. Moreover, 
a version of the spectral mapping theorem for the distinguished spectrum is pointed out.
\end{abstract}

\bigskip


\bigskip 


\date{January 21, 2021}


\section{ Introduction and preliminaries }

Max algebra, together with  its isomorphic versions, provides an attractive way of describing a class of non-linear problems appearing for instance in manufacturing and transportation scheduling, discrete event-dynamic systems, combinatorial optimization, information technology, graph theory, mathematical physics and DNA analysis (see e.g. \cite{BCOQ92,Bapat,Butkovic,BGC-G09,GOW19,GPZ20,atwork,Litvinov07,LM05,MN02,MP15,PS05}, and the references cited there). The usefulness of max algebra arises from a fact that these non-linear problems become linear when described in max algebra terminology.

Max algebra is a semialgebra over the ordered, idempotent semifield
$\mathbb{R}_+$, equipped with the operations of addition $a
\oplus b = \max(a,b)$ and classical multiplication $ab$. 
As in standard arithmetic, the operations of addition and multiplication are associative and commutative, and multiplication is distributive over addition.
Matrix and polynomial operations are defined similarly to their standard counterparts, with the max operation replacing the standard summation.
In particular, for two $n \times n$ nonnegative matrices $A$ and $ B $ 
the max algebra product $A\otimes B$ is defined by
$$(A\otimes B) _{ij} = \max _{l=1,\ldots , n} a_{il} b_{lj} $$
for all $i,j =1, \ldots , n$  (the max algebra product  $A\otimes B$ differs from the Kronecker product of matrices and the notation $\otimes$ is quite standard in max algebra). The $m$th  power in max algebra of $A$
is denoted by $A^m _{\otimes}$. More precisely,
$$[A^m _{\otimes}] _{ij}  =  \max _{ i_1 , \ldots , i_{m-1} \in \{1,\ldots , n\}} a_{i i_1}a_{i_1 i_2}\cdots a_{i_{m-1} j}$$
for all $i,j =1, \ldots , n$ (and thus $[A^m _{\otimes}] _{ij}$ equals the heaviest path from $i$ to $j$ in a suitably defined graph associated to $A$, for details see e.g. \cite{Butkovic,Schneider, MP15}).  

Max algebra is isomorphic to max-plus algebra, which is the semifield $\mathbb{R} \cup \{-\infty \}$, where addition is replaced by maximum and multiplication  by addition, and also to min-plus algebra $\mathbb{R} \cup \{\infty \}$, where addition is replaced by minimum and multiplication by addition. The term  tropical algebra is sometimes used for any of the above isomorphic semifields.  Tropical algebra is a part of a broader branch of mathematics,
called  ``idempotent mathematics'' 
(see e.g. \cite{Litvinov07,LM05} and the references cited there). For other infinite dimensional generalizations see e.g. \cite{AGW,MN02,MP17,MP18,P06}.

\medskip

It is well known that spectral theory for geometric eigenvalues in max-algebra is closely related to Perron-Frobenius spectral theory for classical distinguished eigenvalues (corresponding to classical nonnegative  eigenvectors) of a given nonnegative matrix (see e.g. \cite{Butkovic,BGC-G09,Cuninghame,Gaubert,Schneider,MP15}). In the current article  we prove further results  on  (geometric) eigenvalues in max-algebra and on classical distinguished eigenvalues. Moreover, we point out in Section 2 some apparently not known explicit asymptotic formulae between both kinds of eigenvalues, which can be used for transferring results from one setting to the other (and vice versa). 
In Section 3 we prove certain apparently new inequalities for both kinds of eigenvalues of Hadamard (Schur) products of matrices.

 Let $A=(a_{ij})$, $B=(b_{ij})\in \zR^{n\times n}$. We write $A\geq B$, if $a_{ij} \ge b_{ij}$ for every $1\le i, j\le n$. A matrix $A$ is  called nonnegative if $a_{ij}\geq 0$. The set of all $n \times n$ nonnegative matrices is denoted by $\zR_+^{n\times n}$.

For $A, B\in \zR_+^{n\times n}$ and $x\in \zR_+^{n}$, the products and sum in max algebra are defined as follows: 
\[
(A\otimes B)_{ij}= \max_{k} a_{ik} b_{kj}, \; \;\; (A\otimes x)_i = \max _j a_{ij} x_j, \;\; (A\oplus B)_{ij}= \max \{a_{ij}, b_{ij}\}.
\]
We also denote
\[
\|A\|= \max _{i,j} a_{ij},~~~ \|x\|= \max _{i} x_{i}.
\]  
By $A^j _{\otimes}$ we denote the $j$th max power of $A$. By $AB$, $Ax$, $A^j$ we denote the standard products and powers, respectively. 

Let $A\circ B$ denote the Hadamard (Schur) product of matrices and $A^{(t)}$ be the Hadamard (Schur) power for $t>0$, i.e.,
$[A\circ B]_{ij}=a_{ij}b_{ij}$ and $A^{(t)}=[a_{ij} ^t]$. Similarly we denote the Hadamard (Schur) power $x^{(t)}$ of a vector $x \in   \zR_+ ^{n}$.

The role of the spectral radius of $A\in \zR^{n\times n} _+$ in max algebra is played by the maximum cycle geometric mean $r_{\otimes}(A)$, which is defined by
\[
r_{\otimes}(A)=\max\Bigl\{(a_{i_1 i_k}\cdots a_{i_3 i_2}a_{i_2 i_1})^{1/k}: k \in \zN \;\;\mathrm{and}\;\; i_1,\dots,i_k\in\{1,\dots,n\}\Bigr\},\]
and it is equal to

\[r_{\otimes}(A)=\]\[\max\Bigl\{(a_{i_1 i_k}\cdots a_{i_3 i_2}a_{i_2 i_1})^{1/k}: k \le n, \; i_1,\dots,i_k\in\{1,\dots,n\}~ 
\mbox{ are distinct} 
\Bigr\}.\]
Recall that also a Gelfand type formula holds (see e.g. \cite{MP15} and the references cited there):
\[
r_{\otimes}(A)= \lim _{j \to \infty} \|A^j _{\otimes}\|^{1/j} =  \inf _{j \in \zN} \|A^j _{\otimes}\|^{1/j}.
\]
Let $\sigma _{\otimes} (A)$ denote the set of (geometric) eigenvalues of $A$ in max algebra (the point spectrum in max algebra), i.e.,
$$\sigma _{\otimes} (A)=\{\lambda \ge 0 : A \otimes x = \lambda x \;\; \mathrm{for}\;\;\mathrm{some} \;\; x\ge 0, x\neq 0\}.$$ 
It is known that
\[r_{\otimes}(A) = \max \{\lambda : \lambda \in \sigma _{\otimes} (A)\}\]
and if $A$ is irreducible, then $r_{\otimes}(A) >0$ is the only eigenvalue of $A$ in max algebra and all the corresponding max algebra eigenvectors are strictly positive.

For $A\in \zR_+ ^{n\times n}$ and $x \in \zR_+ ^n$ let $r_{x}(A)$  denote the local spectral radius of $A$ at $x$ in max algebra, i.e., 
\begin{equation}
	\label{local}
	r_x(A)= \lim _{k\to \infty } \|A^k _{\otimes} \otimes x\|^{1/k}.
\end{equation}
It was proved in \cite[Theorem 2.4]{MP15} that the limit in (\ref{local}) exists and that
\begin{equation}
	\label{local_max}
	r_x(A) = \max \{r_{e_j} (A): x_j \neq 0, j=1, \ldots , n\},
\end{equation}
where $e_j$ denotes the $j$th standard basis vector for $j=1, \ldots , n$.

The following result from \cite[Lemma 2.1]{MP15} describes $r_{e_j}(A)$ for all $j\in\{1,\dots,n\}$. 

\begin{lem}
	\label{vladimir}
	Let $A\in \mathbb{R}_+ ^{n\times n}$, $j\in\{1,\dots,n\}$. Then $r_{e_j}(A)$ is the maximum of all $t\ge 0$ with the property {\rm (*)}:
	\label{cycles}
 there exist $c\ge 0$, $d\ge 1$ and mutually distinct indices
$i_0:=j, i_1,\dots,i_c,i_{c+1},
\dots,i_{c+d-1}\in\{1,\dots,n\}$ such that
	$$ \prod_{s=0}^{c-1}a_{i_{s+1},i_{s}}\ne 0 \;\;\; \;\;\;  \mathrm{and}\;\;\;\;\;\;  \prod_{s=c}^{c+d-1}a_{i_{s+1},i_{s}}=t^d,$$
	where we set $i_{c+d}=i_c$.
\end{lem}

In \cite[Theorem 2.7]{MP15} the following result of Gunawardena (\cite[Proposition 2.1]{Gunawardena}) was  reproved in a more linear algebraic fashion. 
\begin{thm}\label{Spec1}
	If $A\in \zR_+ ^{n\times n}$, then 
	$$
	\sigma_ {\otimes}(A)=\{\lambda: \mbox{ there exists } j\in\{1,\dots,n\}, \lambda =r_{e_j}(A)\}.
	$$
\end{thm}
An analogue holds also for classical distinguished eigenvalues  of a nonnegative matrix. For $A\in \zR_+ ^{n\times n}$ and $x \in \zR_+ ^n$ let $\rho_{x}(A)$  denote the (classical) local spectral radius of $A$ at $x$, i.e.,  
\begin{equation}
	\label{local_def}
	\rho_{x} (A)= \limsup _{k \to \infty }\|A^k   x\|^{1/k}. 
\end{equation}
The distinguished part of the spectrum $\sigma _D (A)$ of $A$ is the set of distinguished eigenvalues of $A$, i.e.,
$$\sigma _{D} (A)=\{\lambda \ge 0 : A x = \lambda x \;\; \mathrm{for}\;\;\mathrm{some} \;\; x\ge 0, x\neq 0\}.$$ 
It holds (see  \cite[Theorem 2.16]{MP15}) that
\begin{equation}
	\label{dist_eig}
	\sigma_ {D}(A)=\{\lambda: \mbox{ there exists } j\in\{1,\dots,n\}, \lambda =\rho_{e_j}(A)\}.
\end{equation}

In Theorems \ref{Schurt}, \ref{sup_inf_t_ji}, \ref{max_potence} and \ref{MTHM1} (which are some of the main results of the article) we prove asymptotic formulae between
(geometric) eigenvalues in max-algebra and  classical distinguished eigenvalues. To be more precise, in Theorems \ref{Schurt} and \ref{sup_inf_t_ji} we show that for 
$A\in \mathbb{R}_+^{n\times n}$ and each  $i=1, ..., n $, the mapping $t\mapsto \rho _{e_i } (A^{(t)}) ^{1/t}$ is decreasing in $t\in (0, \infty)$ and we have

	$$r _{e_i }(A) =\lim _{t\to \infty } \rho _{e_i } (A^{(t)}) ^{1/t}=\inf _{t \in (0,\infty)} \rho _{e_i } (A^{(t)}) ^{1/t}=\sup_{t \in (0,\infty)}(n^{-1} \rho _{e_i } (A^{(t)})) ^{1/t} . $$
	{\color{black} In Theorem  \ref{max_potence} we show that for each   $i=1, ..., n $ we also have \label{max_potence}
	$$r _{e_i }(A) =\lim _{k \to \infty } \rho _{e_i } (A^{k} _{\otimes}) ^{1/k}  = \inf _{k \in \mathbb{N}} \rho _{e_i } (A^{k} _{\otimes}) ^{1/k}=\sup_{k \in \mathbb{N}}(n^{-1} \rho _{e_i } (A^{k} _{\otimes}) ^{1/k} .$$
	In Theorem \ref{MTHM1} we prove that for each $i \in \{1, \ldots , n\} $  it holds also that
$$\rho _{e_i }(A) =\lim _{k\to \infty }  r_{e_i } (A^{k}) ^{1/k}= \sup _{k \in \mathbb{N}}  r_{e_i } (A^{k}) ^{1/k}= \inf _{k \in \mathbb{N}} ( n r_{e_i } (A^{k}) )^{1/k} .$$
	}

\bigskip

The following results on eigenvalues of multivariable polynomials in max algebra and on distinguished eigenvalues of multivariable polynomials were proved in \cite[Theorem 3.6]{Schneider}. These results may be considered as a generalization of the spectral mapping theorem for max polynomials (see also \cite[Theorem 3.4]{MP15}) and a generalization of distinguished polynomial spectral mapping theorem. The proofs of these results are based on the fact that
commuting matrices in max algebra have a joint max-eigenvector and on an analogue of this result for distinguished eigenvectors (\cite[Theorem 3.5]{Schneider}).

\begin{thm}\label{spower}  Let $A_1, \ldots , A_m \in \zR _+ ^{n\times n}$ commute in pairs in max algebra and let  
	$p_{\otimes }(x_1, \ldots , x_m) $ be a max polynomial (of $m$ variables). Then the following properties hold.
	
	\begin{enumerate}
		\item For each $i \in \{1, \ldots , m\}  $ and $ \lambda _i \in \sigma _{\otimes} (A_i)$ there exists  $ \lambda _j \in \sigma _{\otimes} (A_j)$ for all 
		$j \in \{1, \ldots , m\}  $, $j \neq i$,  such that 
		$p_{\otimes } (\lambda _1, \ldots , \lambda _m)  \in \sigma _{\otimes } (p_{\otimes } (A _1, \ldots , A_m) )$.
		\medskip
		
		\item For each $\lambda \in  \sigma _{\otimes } (p_{\otimes } (A _1, \ldots , A_m) )$ there exist $\lambda _i \in \sigma _{\otimes} (A_i)$ for all 
		$i \in \{1, \ldots , m\}  $ such that $\lambda =  p_{\otimes } (\lambda _1, \ldots , \lambda _m)$.
	\end{enumerate}
\end{thm}

\begin{thm}\label{dist_power}  Let $A_1, \ldots , A_m \in \zR _+ ^{n\times n}$ commute in pairs  and let  
	$p(x_1, \ldots , x_m) $ be a  real polynomial (of $m$ variables) such that 
	$p (A _1, \ldots , A_m) $ is nonnegative. 
	Then the following properties hold.
	
	\begin{enumerate}
		\item For each $i \in \{1, \ldots , m\}  $ and $ \lambda _i \in \sigma _{D} (A_i)$ there exists  $ \lambda _j \in \sigma _{D} (A_j)$ for all 
		$j \in \{1, \ldots , m\}  $, $j \neq i$,  such that 
		$p (\lambda _1, \ldots , \lambda _m)  \in \sigma _{D} (p (A _1, \ldots , A_m) )$.
		\medskip
		
		\item For each $\lambda \in  \sigma _{D} (p (A _1, \ldots , A_m) )$ there exist $\lambda _i \in \sigma _{D} (A_i)$ for all 
		$i \in \{1, \ldots , m\}  $ such that $\lambda =  p(\lambda _1, \ldots , \lambda _m)$.
	\end{enumerate}
\end{thm}

Below we prove results which give additional information to Theorems \ref{spower} and \ref{dist_power}.
They follow from results of \cite{Schneider}    
and from  descriptions of $r_{e_i}$ and $\rho_{e_i}$ via access relations (see e.g. \cite{Schneider}  and \cite[Corollary 2.9]{MP15}). Recall that by \cite[Corollaries 4.2 and 4.2A]{Schneider} for each $i$ the max eigenvalue $r_{e_i}(A)$ is the maximum cycle geometric mean of some spectral class (in max algebra) and that the distinguished eigenvalue $\rho_{e_i}(A)$ equals the Perron root of some premier spectral class (in nonnegative linear algebra).  For unexplained details on Frobenius normal form and access relations we refer the reader to \cite[Section 4]{Schneider} and \cite{MP15}.

\begin{thm}\label{powers}  Let $A_1, \ldots , A_m \in \zR _+ ^{n\times n}$ commute in pairs in max algebra and let  
	$p_{\otimes }(x_1, \ldots , x_m) $ be a max polynomial (of $m$ variables).
	{\color{black}Then for  each $i\in \{1, \ldots , n\}$  there exists $k\in \{1, \ldots , n\}$ such that the equality 
	
	\begin{equation}\label{pw1}
		r_{e_i}(p_{\otimes}(A_1,\dots,A_m))=p_{\otimes}(r_{e_k}(A_1),\dots,r_{e_k}(A_m))
	\end{equation}
	holds if one of the following conditions is satisfied:}
	
	(i) if at least one of the matrices $A_1, \ldots , A_m, p_{\otimes}(A_1,\dots,A_m)$ is irreducible. In this case we may take $k=i$;
	
	{\color{black}(ii) if all classes of $A_j$, for each $j=1,\ldots , m$, have distinct eigenvalues in max algebra (if cycle geometric means of all the classes are distinct). }%
\end{thm}

\begin{proof} %
	Case (i): Since  $A_1, \ldots , A_m, p_{\otimes}(A_1,\dots,A_m) $ commute in pairs in max algebra, 
	then by 
	\cite[Lemma 4.6]{Schneider} 
	each of these matrices has a unique eigenvalue in max algebra and so by  Theorem \ref{spower} we have 
	\begin{equation}
		r_{\otimes}(p_{\otimes}(A_1,\dots,A_m))=p_{\otimes}(r_{\otimes}(A_1),\dots,r_{\otimes}(A_m)),
	\end{equation}
	which is in this case equivalent to (\ref{pw1}) (see also \cite[Corollary 4.6 and Remark 4.7]{Schneider}).

	Case (ii): Under the assumptions in this case the spectral classes of matrices $ A_1, \ldots , A_m $  coincide by  \cite[Theorem 4.8(iii)]{Schneider}.
	 Now we apply the arguments from the proof of \cite[Theorem 4.8(iv)]{Schneider}. Let $i\in \{1, \ldots , n\}$.  Since the matrices $A_1, \ldots , A_m, p_{\otimes}(A_1,\dots,A_m)$
	 commute in pairs in max algebra there exists by \cite[Theorem 3.5]{Schneider} a max eigenvector $v$ satisfying 
	 $$p_{\otimes}(A_1,\dots,A_m)\otimes v=r_{e_i}(p_{\otimes}(A_1,\dots,A_m))v,$$
	  which is also a max eigenvector for all $ A_1, \ldots , A_m $. 
	 By \cite[Corollary 4.2(v)]{Schneider} there exists a  common spectral class $\mu$ of  matrices $ A_1, \ldots , A_m $ such that  the support of $v$ is equal to the initial segment generated by $\mu$. By \cite[Corollary 4.2(ii)]{Schneider} and  Lemma \ref{vladimir}, {\color{black} for each $k \in \mu$ we have
	 $A_j \otimes v =r_{e_k}(A_j)v$ for all $j=1, \ldots , m$. } 
	 {\color{black} It follows that 
	 $$r_{e_i}(p_{\otimes}(A_1,\dots,A_m))v=p_{\otimes}(A_1,\dots,A_m)\otimes v=p_{\otimes}(r_{e_k}(A_1),\dots,r_{e_k}(A_m))v $$
	 and thus (\ref{pw1}) holds.}
\end{proof}

The following results is an analogue of \cite[Corollary 2.9]{MP15} for distinguished eigenvalues. It follows from (\ref{dist_eig}) and \cite[Corollary 4.2A]{Schneider}. 

\begin{thm} Let $A \in \zR _+ ^{n\times n}$ and $i \in \{1, \ldots , n\}$. Then $\rho _{e_i} (A)$ equals the maximum of  Perron roots of classes $\mu$ of $A$ such that 
$\mu$ has access to $i$.
\label{access_dist}
\end{thm}
\begin{proof} Assume without loss of generality that $A$ is in the (lower block triangular) Frobenius normal form. Observe that $\rho _{e_i} (A) \ge \rho _{e_j} (A)$ if $a_{ji} >0$. By (\ref{dist_eig}) and \cite[Corollary 4.2A]{Schneider} there exists a premier spectral class $\mu$ with Perron root  $\rho _{e_i} (A)$ such that $\mu$ has access to $i$. 
This completes the proof.
\end{proof}

By applying  \cite[Theorem 3.5, Theorem 4.8A, Corollary 4.2A, Lemma 4.6A]{Schneider}) and Theorem \ref{access_dist} the following theorem is proved in a similar way as Theorem \ref{powers}.

\begin{thm}\label{powers_nonneg}  Let $A_1, \ldots , A_m \in \zR _+ ^{n\times n}$ commute in pairs and let  
	$p(x_1, \ldots , x_m) $ be a real polynomial (of $m$ variables)  such that 
	$p (A _1, \ldots , A_m) $ is nonnegative. %
	
	{\color{black}Then for  each $i\in \{1, \ldots , n\}$  there exists $k\in \{1, \ldots , n\}$ such that the equality }
	\begin{equation}\label{pw1_nonneg}
		\rho_{e_i}(p(A_1,\dots,A_m))=p(\rho_{e_k}(A_1),\dots, \rho_{e_k}(A_m))
	\end{equation}
	holds if one of the following conditions is satisfied:
	
	(i) if at least one of the matrices $A_1, \ldots , A_m, p(A_1,\dots, A_m)$ is irreducible. In this case we may take $k=i$;
	
	{\color{black} (ii) if all classes of $A_j$, for each $j=1,\ldots , m$, have distinct Perron roots. }%
\end{thm}
The following useful result gives more information in the special case of powers of a nonnegative matrix. It follows from (\ref{dist_eig}), \cite[Corollary 4.2A]{Schneider}, Theorem \ref{access_dist} and \cite[Theorem 5.4 and Lemma 5.3]{BSST13}. Observe that in \cite{BSST13} the premier spectral classes (in nonnegative linear algebra) were simply called spectral classes (in nonnegative linear algebra).
\begin{thm}  Let $A \in \zR _+ ^{n\times n}$, $i \in \{1, \ldots , n\}$ and $m\in \mathbb{N}$. Then
\begin{equation}
\rho _{e_i} (A^m)= \rho _{e_i} (A)^m .
\label{powers_rho_i}
\end{equation}
\end{thm} 
\begin{proof}
The inequality
	\begin{equation}
		\label{good2}
		\rho _{x }(A^m)\le \rho _{x}(A)^m 
	\end{equation}
	for $x \in \zR _+ ^{n}$ is well known (see e.g. the proof of \cite[Proposition 2.1]{MN02}) and easy to establish. Indeed, by (\ref{local_def}) we have 
	$$\rho_{x} (A^m)= \limsup _{k \to \infty }\|A^{mk}   x\|^{\frac{1}{k}}=(\limsup _{k \to \infty }\|A^{mk}   x\|^{\frac{1}{mk}})^m \le \rho _{x}(A)^m . $$
	To prove (\ref{powers_rho_i}) we need to show that 
	\begin{equation}
\rho _{e_i} (A^m)\ge \rho _{e_i} (A)^m .
\label{powers_rho_i2}
\end{equation}
	We may assume that $ \rho _{e_i} (A) >0$. By (\ref{dist_eig}) and \cite[Corollary 4.2A]{Schneider} there exists a premier spectral class $\mu$ of $A$ with Perron root  $\rho _{e_i} (A)$ such that $\mu$ has access to $i$ in $A$. By \cite[Theorem 5.4 and Lemma 5.3]{BSST13}, $\mu$ is an ancestor of some premier spectral class $\nu$ of $A^m$ with  Perron root  $\rho _{e_i} (A)^m$, such that $\nu \subset \mu$ and such that $\nu$ has access to $i$ in $A^m$. By Theorem \ref{access_dist} inequality (\ref{powers_rho_i2}) holds, which completes the proof.
\end{proof}

In Theorem 
\ref{gen_matr_sp_map}
we prove a version of the spectral mapping theorem for power series $f$ for the distinguished part of the spectrum by showing that under suitable conditions 
			$\sigma _D(f(A))=f (\sigma _D(A))$ holds.

\bigskip

In the literature, inequalities on the classical spectral radius of matrices constantly attract substantial attention. 
For example, it is well known that for non-negative $n \times n$ matrices $A$ and $B$, the spectral radius $\rho(A \circ B)$ of the Hadamard (Schur) product satisfies
\begin{equation}
	\label{Els_ineq}
	\rho(A \circ B) \leq \rho(A) \rho (B) 
\end{equation}
and 
\begin{equation}
	\label{Aud_ineq}
	\rho(A \circ B) \leq \rho(AB),
\end{equation}
where $AB$ denotes the usual product of matrices $A$ and $B$ (see e.g. \cite{Audenaert,Elsner,HoJo1}). Relatively recently, several closely related inequalities for the spectral radius have been established (see e.g. \cite{CZ,DP16,HZ10,Huang,P12,P17,P18a,P18b,Sch10a,Sch10b,Z}). 
It is also known that the analogues of the above and several related inequalities hold also in max algebra (see e.g. \cite{MP18,P12,RLP19}). 

In a very special case of results of Section 3 of the current article we establish that the analogues of (\ref{Els_ineq}) for all  (geometric) eigenvalues in max-algebra and all classical distinguished eigenvalues of nonnegative matrices are valid. On the contrary, we show that the analogues of (\ref{Aud_ineq})  for both types of eigenvalues are not correct in general, but we do establish some new closely related inequalities.

\section{Asymptotic relations between eigenvalues in max algebra and distinguished classical eigenvalues }

Let $A \in \zR _+ ^{n\times n}$ and 
let $\rho (A)$ denote the usual spectral radius of $A$. The Gelfand formula states that 
$$\rho (A)= \lim _{k \to \infty }\|A^k \|^{1/k}.$$
It is also well known (see e.g.  \cite{Bapat,Elsner,P08,P11}) that 
\begin{equation}
	\label{good_n_ineq}
	r_{\otimes}(A) \le \rho (A) \le n r_{\otimes} (A).
\end{equation}

Recall that for $t >0$,  $A^{(t)}=[a_{ij} ^t]$ denotes the (entrywise) Hadamard (Schur) power.
From (\ref{good_n_ineq}) and the fact that $r_{\otimes}(A^{(t)}) =r_{\otimes} (A)^t$ for $t>0$ 
the following known equality follows  (see e.g. \cite{Bapat,Elsner,P08,P11})
\begin{equation}
	\label{Hadam_t}
	r_{\otimes}(A) =\lim _{t\to \infty } \rho (A^{(t)}) ^{1/t}. 
\end{equation}

Algebraic eigenvalues of $A$ in max algebra are the tropical roots of its max algebraic characteristic polynomial (see e.g. \cite{AGB01,Butkovic,RLP19} for detailed definitions). There are exactly $n$ algebraic max eigenvalues (counting suitable multiplicities).  It is known that in max algebra all (geometric) eigenvalues are also algebraic eigenvalues (but not vice versa in general). 
The following result that improves (\ref{Hadam_t}) is a restatement of a well known max-plus algebra result from \cite{AGB01}. 




\begin{thm}
	Let $r_1 \le \cdots \le r_n$ be algebraic eigenvalues of  $A\in \zR_ +^{n\times n}$ in max algebra.
	Then for each $i=1, ..., n $ we have
	\begin{equation}
		\label{AGB}
		r_i = \lim _{t\to \infty} |\lambda _i (t)|^{1/t},
	\end{equation}
	where $\lambda _i (t)$ are the (classical linear algebra) eigenvalues  of $A^{(t)}$ with  $|\lambda _1 (t)|\le \cdots \le |\lambda _n (t)| $.
\end{thm} 


%

The following useful result for Hadamard powers 
is known (see e.g. \cite{P08})  and easily verified by definitions.

\begin{lem}\label{Lem1} Let $A_1, \cdots, A_m, A \in   \zR_+^{n\times n}$, $ x\in  \zR_+^{n}$, $t>0$ and $k \in \zN$. \\
	Then $\|A^{(t)}\|= \|A\|^t$,  $\|x^{(t)}\|= \|x\|^t$,
	\begin{equation}
		\label{t_prod}
		A_1 ^{(t)} \otimes \cdots \otimes A_m ^{(t)}= (A_1 \otimes  \cdots \otimes A_m)^{(t)}.
	\end{equation}
	and so
	\[(A ^{(t)})^k _{\otimes} = (A^k _{\otimes})^{(t)}.\] 
	
\end{lem}

Next we show that in (\ref{AGB}) the (geometric) eigenvalues of $A$ are determined by distinguished (classical) eigenvalues of $A^{(t)}$.

\begin{thm}
	\label{Schurt}
	Let $A\in \mathbb{R}_+^{n\times n}$. Then for each  $i=1, ..., n $ we have
	$$r _{e_i }(A) =\lim _{t\to \infty } \rho _{e_i } (A^{(t)}) ^{1/t}. $$
\end{thm}
\begin{proof}
	For each $m \in \mathbb{N}$, $t>0$ and  $i=1, ..., n $ we have
	$$\|(A^{(t)})^m _{\otimes} \otimes e_i\| \le \|(A^{(t)})^m  e_i\| \le n^{m-1} \|(A^{(t)})^m  _{\otimes} \otimes e_i\| .$$
	By (\ref{t_prod}) it holds $\|(A^{(t)})^m _{\otimes} \otimes e_i\| = \|(A^m _{\otimes})^{(t)} \otimes e_i\|=\|A^m _{\otimes} \otimes e_i\|^t $. Consequently,
	$$\|A^m _{\otimes} \otimes e_i\|^t \le \|(A^{(t)})^m  e_i\| \le n^{m-1} \|A^m _{\otimes} \otimes e_i\|^t .$$
	Taking the $m$th root and letting $m \to \infty $ it follows
	$$r _{e_i }(A) ^t \le  \rho _{e_i } (A^{(t)}) \le  n r _{e_i }(A) ^t .$$
	Now taking the $t$th root and letting $t \to \infty$ completes the proof.
\end{proof}
In fact, we have established in the proof above also the following result
(which also follows from results of \cite{MP15}).
\begin{prop}\label{radius1} Let  $A \in \zR_+^{n\times n}$, $t>0$ and $i \in \{1, \ldots , n\}$. Then
	\begin{equation}
		\label{Sch_t}
		r_{e_i} (A^{(t)}) = r_{e_i} (A)^t
	\end{equation}
	and
	\begin{equation}
		\label{ineq_n}
		r_{e_i} (A) \le  \rho_{e_i} (A) \le n  r_{e_i} (A).
	\end{equation}
\end{prop}

Let $A_1, \ldots , A_m$ and 
$t \ge 1$. It was proved in \cite[Lemma 4.2]{P06} that 

\begin{equation}
	A_{1} ^{(t)} \cdots  A_{m} ^{(t)} \le ( A_1 \cdots A_m )^{(t)}.
	\label{t_dobro}
\end{equation}

\begin{lem} Let $A\in \mathbb{R}_+^{n\times n}$, $t\ge 1$.  Then for each and $i=1, ..., n $ we have 
	$$ \rho _{e_i } (A^{(t)}) \le  \rho _{e_i } (A)^t.$$
	\label{lemma_rhot}
\end{lem}
\begin{proof} The result follows from (\ref{t_dobro}), since
	$$\rho _{e_i} (A^{(t)}) = \limsup _{k \to \infty} \|(A^{(t)})^k e_i\|^{1/k}\le  \limsup _{k \to \infty} \|(A^k)^{(t)} e_i\|^{1/k}
	=\rho _{e_i} (A)^t .$$
	Observe that the last equality follows from $\|(A^k)^{(t)} e_i\|= \|A^k e_i\|^t $, which is valid since $\|A^k e_i\|$ equals the maximal entry of the $i$th column of $A^k$ and since $\|(A^k)^{(t)} e_i\|$ equals the maximal entry of the $i$th column of $(A^k)^{(t)}$. 
\end{proof}
The following result, which is motivated by results and proofs of \cite{P08} and \cite{P11} follows from Theorem \ref{Schurt} and (\ref{ineq_n}).
\begin{thm}
\label{sup_inf_t_ji}
	Let $A\in \mathbb{R}_+^{n\times n}$ and $i=1, ..., n $. Then $\rho _{e_i } (A^{(t)}) ^{1/t}$ is decreasing in $t\in (0, \infty)$ and we have
	\begin{equation}
		\label{inf_t}
		r _{e_i }(A) =\inf _{t \in (0,\infty)} \rho _{e_i } (A^{(t)}) ^{1/t} 
	\end{equation}
	and 
	\begin{equation}
		\label{sup_t}
		r _{e_i }(A) =\sup_{t \in (0,\infty)}(n^{-1} \rho _{e_i } (A^{(t)})) ^{1/t} .
	\end{equation}
	
\end{thm}
\begin{proof} Let $0 <t \le s$. Then by  Lemma \ref{lemma_rhot} we have
	$$ \rho _{e_i } (A^{(s)})^{1/s}= \rho _{e_i } ((A^{(t)})^{(\frac{s}{t})})^{1/s} \le  \rho _{e_i } (A^{(t)})^{1/t}$$
	and so $\rho _{e_i } (A^{(t)}) ^{1/t}$ is decreasing in $t\in (0, \infty)$. Now (\ref{inf_t}) follows from Theorem \ref{Schurt}. Similarly, by (\ref{ineq_n}) we have
	$n^{-1} \rho _{e_i } (A^{(t)}) \le r_{e_i} (A^{(t)})=r_{e_i} (A)^t$. By 
	 Theorem \ref{Schurt}  it follows
	 $$r _{e_i }(A) =\lim _{t\to \infty } (n^{-1}\rho _{e_i } (A^{(t)}) )^{1/t}\le \sup_{t \in (0,\infty)}(n^{-1} \rho _{e_i } (A^{(t)})) ^{1/t} \le r_{e_i} (A), $$
 which	proves (\ref{sup_t}).
\end{proof}
The following results is proved in a similar manner.
\begin{thm}
	\label{max_potence}
	Let $A\in \mathbb{R}_+^{n\times n}$ and $i=1, ..., n $. Then 
	\begin{equation}
		\label{lim_and_inf}
		r _{e_i }(A) =\lim _{k \to \infty } \rho _{e_i } (A^{k} _{\otimes}) ^{1/k}  = \inf _{k \in \mathbb{N}} \rho _{e_i } (A^{k} _{\otimes}) ^{1/k} 
	\end{equation}
	and 
	\begin{equation}
		\label{sup_k}
		r _{e_i }(A) =\sup_{k \in \mathbb{N}}(n^{-1} \rho _{e_i } (A^{k} _{\otimes}) ^{1/k} .
	\end{equation}
	
\end{thm}
\begin{proof}
	Let $k\in \mathbb{N}$.
	It follows from (\ref{local}) that $r_{e_i} (A^k _{\otimes}) = r_{e_i} (A)^k$. By (\ref{ineq_n}) we have
	$$r_{e_i} (A)^k \le  \rho_{e_i} (A^k _{\otimes}) \le n  r_{e_i} (A)^k.$$
	By taking the $k$th root and letting $k \to \infty$ we obtain $r _{e_i }(A) =\lim _{k \to \infty } \rho _{e_i } (A^{k} _{\otimes}) ^{1/k}$.
	Moreover, since
	$$r_{e_i} (A) \le \inf _{k \in \mathbb{N}} \rho _{e_i } (A^{k} _{\otimes}) ^{1/k} \le \lim _{k \to \infty } \rho _{e_i } (A^{k} _{\otimes}) ^{1/k}= r_{e_i} (A) $$ 
	and
	$$r_{e_i} (A) \ge \sup_{k \in \mathbb{N}}(n^{-1} \rho _{e_i } (A^{k} _{\otimes}) ^{1/k}\ge \lim _{k \to \infty } (n^{-1}\rho _{e_i } (A^{k} _{\otimes}) )^{1/k}= r_{e_i} (A) $$ 
	it follows that 
	$$r _{e_i }(A) = \inf _{k \in \mathbb{N}} \rho _{e_i } (A^{k} _{\otimes}) ^{1/k}=\sup_{k \in \mathbb{N}}(n^{-1} \rho _{e_i } (A^{k} _{\otimes})) ^{1/k},$$
	which completes the proof.
\end{proof}


It was proved in \cite[Theorem 16]{Bapat} that for  $A\in \mathbb{R}_+^{n\times n}$ and classical powers $A^k$ we have
\begin{equation}
	\label{Bap}
	\rho (A)= \lim _{k \to \infty} r_{\otimes} (A^k) ^{1/k}.
\end{equation}
Observe that it was also shown in \cite{Bapat} that $r_{\otimes} (A^k) ^{1/k}$ is in general not increasing in $k$.

Similarly as in the proof of Theorem \ref{max_potence} one can prove that the following result follows from (\ref{Bap}) and  (\ref{good_n_ineq}). 
\begin{cor} If $A\in \mathbb{R}_+^{n\times n}$, then 
	\begin{equation}
		\rho (A)= \sup _{k \in \mathbb{N}} r_{\otimes} (A^k) ^{1/k}= \inf _{k \in \mathbb{N}} (n r_{\otimes} (A^k) )^{1/k}.
	\end{equation}
\end{cor}

 The following 
result 
is an analogue of the above result for all $\rho _{e_i }$. It 
gives an asymptotic formula on
distinguished  (classical) eigenvalues of $A$ calculated from the  (geometric) eigenvalues in max algebra of its classical powers.

\begin{thm}\label{MTHM1}
Let $A\in \mathbb{R}_+^{n\times n}$ and $i \in \{1, \ldots , n\} $. Then 
	\begin{equation}
		\label{help}
		\rho _{e_i }(A) =\lim _{k\to \infty }  r_{e_i } (A^{k}) ^{1/k}= \sup _{k \in \mathbb{N}}  r_{e_i } (A^{k}) ^{1/k}= \inf _{k \in \mathbb{N}} ( n r_{e_i } (A^{k}) )^{1/k} .
	\end{equation}

\end{thm}

%

\begin{proof}
	Let $k \in \mathbb{N}$. It follows from (\ref{powers_rho_i}) and (\ref{ineq_n})
	\begin{equation}
		\label{classical_ineq}
		\rho _{e_i }(A)^k =\rho _{e_i }(A^k) \le n r _{e_i } (A^k) \le  n \rho _{e_i }(A^k)= n \rho _{e_i }(A)^k .
	\end{equation}
	Inequalities (\ref{classical_ineq}) imply (\ref{help}) by applying similar arguments as in the proof of Theorem \ref{max_potence}. 
\end{proof}

A special case  of the following lemma  has already been stated and applied in the proof of \cite[Theorem 2.16]{MP15}.   

\begin{lem}
	Let $A\in \mathbb{R}_+^{n\times n}$ and $x \in \zR_+^n$, $x\neq 0$. Then 
	\begin{equation}
		\rho_x (A) = \max \{\rho _{e_i} (A): x_i \neq 0, i=1, \ldots , n\}
		\label{max_rho_x}
	\end{equation}
\end{lem}
\begin{proof} 
It is straightforward to see that
\[
\max \{\rho _{e_i} (A): x_i \neq 0, i=1, \ldots , n\}\le \rho_x (A).
\]
To prove the reverse inequality let  \[M=\{j:~~x_j\ne 0,~~ 1\le j\le n\}.\]
 Then $x=\sum_{j\in M}x_je_j$ and for
every $k\in\zN$ we have
\begin{equation}\label{SMME}
\|A^kx\|=\|\sum_{j\in M}x_jA^ke_j\|\le \sum_{j\in M}x_j\|
A^ke_j\|.
\end{equation}
Now, suppose that $M$ has $m$ elements $i_j$ for $j=1, \ldots , m$, where $m\le n$. Therefore we have $\{e_{i_1}, e_{i_2}, \cdots, e_{i_m}\}\subseteq \{e_1, e_2,\cdots,  e_n\}$  and so there are $m$ sequences $\{\|A^ke_{i_j}\|\}_{k=1}^{\infty}$. By comparing corresponding elements of these $m$ sequences, there  are integers $P\ge1$ and $1\le s\le m$, and subsets $A_1, A_2,\cdots, A_s$ of $\zN$  such that
\begin{enumerate}
\item each of the sets $A_1, A_2,\cdots, A_s$ has infinitely many elements,
\item $\bigcup_{j=1}^sA_j=\zN-\{1, 2, \cdots, P\}$,
\item for every $k_r\in A_r$ we have $\|A^{k_r}e_{i_j}\|\le \|A^{k_r}e_{i_r}\|$ for all $i_j\neq i_r$.
\end{enumerate}
For $k_r\in A_r$ we have by \eqref{SMME} 
\[
\|A^{k_{r}}x\|\le \sum_{i_j\in M}x_{i_j}\|
A^{k_{r}}e_{i_j}\|\le \left(\sum_{i_j\in M}x_{i_j}\right)\|A^{k_r}e_{i_r}\|,\qquad \mbox{for all}\,\, r=1, 2, \cdots, s.
\]
Taking the $k_{r}$-th root and letting $k_{r} \to \infty $ it follows
\begin{eqnarray*}
\limsup _{k_ {r}\to \infty }\|A^{k_{r}}x\|^{\frac{1}{k_{r}}}& \le & \limsup _{k_ {r}\to \infty }\|A^{k_{r}}e_{i_r}\|^{\frac{1}{k_{r}}}\\
& \le &  \limsup _{k\to \infty }\|A^{k}e_{i_r}\|^{\frac{1}{k}}=\rho _{e_{i_r}} (A)\\
& \le & \max \{\rho _{e_i} (A): x_i \neq 0, i=1, \ldots , n\}.
\end{eqnarray*}
Now, let $\alpha$ be a limit of a subsequence $\|A^{k_l} x\|^{1/{k_l}}$ of $\|A^k x\|^{1/k}$. 
There exist $r\in \{1, \ldots , s\}$ and a subsequence $\{k_{l_r}\}$ such that  $k_{l_r} >P$ and $k_{l_r} \in A_r$ for all $l_r$ (because the number of sets $A_r$ is at most $m$ and each $A_r$  has infinitely many elements).  
Then   \begin{eqnarray*}
\alpha &=& \lim _{k_ {l}\to \infty }\|A^{k_{l}}x\|^{\frac{1}{k_{l}}}
=\lim _{k_ {l_r}\to \infty }\|A^{k_{l_r}}x\|^{\frac{1}{k_{l_r}}}\\
&=&\limsup _{k_ {l_r}\to \infty }\|A^{k_{l_r}}x\|^{\frac{1}{k_{l_r}}} \\ 
& \le & \rho _{e_{i_r}} (A) \le \max \{\rho _{e_i} (A): x_i \neq 0, i=1, \ldots , n\}.
\end{eqnarray*} 
By definition,
\[
 \limsup _{k\to \infty }\|A^{k}x\|^{\frac{1}{k}}=\sup\{\mathrm{all} \;\; \mathrm{subsequential} \;\;  \mathrm{limits} \;\; \mathrm{of} \;\; \|A^{k}x\|^{\frac{1}{k}} \}.
\]
Thus
\[
\rho _{x} (A)\le \max \{\rho _{e_i} (A): x_i \neq 0, i=1, \ldots , n\}.
\]
\end{proof}

	\begin{cor}\label{MCOR1} Let $A\in \mathbb{R}_+^{n\times n}$. Then for all  
		$x \in \zR ^n _+ $  and $m \in \mathbb{N}$ we have  
		\begin{equation}
		\label{power_l}
		\rho _{x }(A^m)=\rho _{x }(A)^m.
		\end{equation}
	\end{cor} 
\begin{proof} The result  is trivial for $x=0$. Let $x\neq 0$. By  
(\ref{max_rho_x}) and (\ref{powers_rho_i}) there exists  $j \in \{1, \ldots , n\}$ such that $x_j >0$ and
	$$\rho _{x}(A)^m = \rho _{e_j }(A)^m= \rho _{e_j }(A^m) \le \rho _{x }(A^m), $$
	which together with (\ref{good2}) proves (\ref{power_l}).
	\end{proof}

Next we point out some continuity properties of $A \mapsto \sigma _{D}(A)$, which correspond to known continuity properties of $A \mapsto \sigma_ {\otimes}(A)$ (see e.g. \cite[Proposition 3.7]{MP15}).
It is known (see \cite[Proposition 3.7(i)]{MP15}) that the spectrum in max algebra $A\mapsto \sigma _{\otimes}(A)$ is upper semi-continuous in $ \mathbb{R}_+ ^{n\times n}$. A similar proof as of \cite[Proposition 3.7(i)]{MP15} also applies to distinguished spectrum. We include the details for the sake of completeness.

\begin{prop}
	
	The distinguished spectrum $A\mapsto \sigma _{D}(A)$ is upper semi-continuous in $ \mathbb{R}_+ ^{n\times n}$.
	
\end{prop}

\begin{proof}
	Let  $A,A_k\in  \mathbb{R}_+ ^{n\times n}$ such that $\|A_k-A\|\to 0$ as $k \to \infty$. Let $\lambda_k\in\sigma _{D}(A_k)$, $\lambda_k\to\lambda$. For each $k\in\mathbb{N}$ there exists an 
	eigenvector $x_k\in \mathbb{R}_+^n$ such that $\|x_k\|=1$ and $A_k x_k=\lambda_k x_k$. By a compactness argument there exists a convergent 
	subsequence of the sequence $(x_k)$. Clearly its limit $x\ge 0$ satisfies $\|x\|=1$ and $A x=\lambda x$, which  completes the proof.
\end{proof}
\begin{rem}
\label{low_c}
 {\rm Moreover, for each $i=1, \ldots, n$ the maps $A\mapsto r_{e_i} (A)$ and 
		$A\mapsto \rho_{e_i} (A)$ are lower-semicontinuous. 
		
		Indeed, it follows from Lemma \ref{vladimir} that $\|A_k-A\|\to 0$ as $k \to \infty$ implies
		$$
		r_{e_i}(A)\le\liminf_{k\to\infty} r_{e_i}(A_k),
		$$
		which establishes the  lower semi-continuity of $r_{e_i}(\cdot)$. Similarly, the lower semi-continuity of $\rho_{e_i}(\cdot)$ follows from the description of distinguished eigenvalues $\rho_{e_i}(A)$ via Frobenius normal form and access relations (Theorem \ref{access_dist}).
	}
\end{rem}

The following example from \cite[Example 3.6]{MP15} shows that  in general, similar to  $A\mapsto \sigma _{\otimes}(A)$, the mappings $A\mapsto \sigma _{D}(A)$, $A\mapsto r_{e_i} (A)$ and 
$A\mapsto \rho_{e_i} (A)$  are not continuous in $ \mathbb{R}_+ ^{n\times n}$.

\begin{ex}
	{\rm Let $A_k=\left[\begin{array}{cc} 1 & 0 \\ k^{-1}& 2\\ \end{array}\right]$, $A=\left[\begin{array}{cc}1 & 0\\ 0 & 2\\ \end{array}\right]$. 
		Then $\sigma _{\otimes}(A_k)= \sigma _{D}(A_k)=\{2\}$ for all $k\in\mathbb{N}$, $\|A_k-A\|\to 0$ as $k \to \infty$ and $\sigma _{\otimes}(A)=\sigma _{D}(A)=\{1,2\}$.
	}
\end{ex}

The following result is an analogue of \cite[Proposition 3.7(ii)]{MP15} for $\sigma_D$. %
\begin{prop}\label{MPORP}  Let $A\in \mathbb{R}_+^{n\times n}$. %
If $x \in \zR ^n _+$, $ A_m\in \mathbb{R}_+^{n\times n}$, $\|A_m-A\|\to 0$ as $m \to \infty$ and $A_1\le A_2\le \cdots$ , then $\rho _{x }(A_m)\to \rho _{x }(A)$ and so $\sigma_D(A_m)\to \sigma_D(A)$ as $m \to \infty$.
\end{prop}
\begin{proof}  By (\ref{max_rho_x}) 
	it suffices to prove the result for $x=e_j$, where $1\le j\le n$. Since $\{A_m\}_{m=1}^{\infty}$ is an increasing sequence of matrices in $\mathbb{R}_+^{n\times n}$, by definition of $\rho_{e_j}(A)$, $\{\rho_{e_j}(A_m)\}_{m=1}^{\infty}$ is an increasing sequence bounded by $\rho_{e_j}(A)$ and so $\lim _{m\to \infty }  \rho_{e_j } (A_m)$ exists. Let $\alpha=\lim _{m\to \infty }  \rho_{e_j } (A_m)$. It is clear that $\alpha\le \rho_{e_j } (A)$. It remains to prove that $ \rho_{e_j } (A)\le \alpha$.
	
	For every positive integer $k$  by (\ref{power_l}) and (\ref{classical_ineq}), we have
	\[
	\begin{aligned}
	\rho_{e_j } (A)^k = \rho_{e_j } (A^k) & \le n r _{e_j } (A^k)\\%
	&= n\lim _{m\to \infty }  r_{e_j } (A_m ^k ) ~~~~~~~~~\qquad \mbox{\cite[ part (ii) Proposition 3.7]{MP15}}\\
	&\le  n\lim _{m\to \infty }  \rho_{e_j } (A_m^k )\\
	&\le  n\lim _{m\to \infty }  \rho_{e_j } (A_m)^k\\
	&=n\alpha^k.
	\end{aligned}
	\]
	Now taking the $k$th root and letting $k \to \infty$ proves $\rho_{e_j } (A)\le \alpha$. Therefore $\rho _{e_j }(A_m)\to \rho _{e_j }(A)$  and so  $\sigma_D(A_m)\to \sigma_D(A)$  as $m \to \infty$ by (\ref{dist_eig}).
\end{proof}
\begin{prop}\label{MTHM2} Assume that $\alpha$ is positive real number. Let $A$ be a positive semidefinite $n\times n$ matrix with nonnegative entries. %
 Then  for every $x \in \zR ^n _+ $  we have
	\[
	\rho _{x }(A^{\alpha})=\rho _{x }(A)^{\alpha}.
	\]
\end{prop}
\begin{proof}
	The result trivially holds  for $x=0$. Assume that $x\neq 0$. The formula holds for $\alpha\in \zN$ by (\ref{power_l}). 
	If $\alpha=\frac{1}{l}$, then
	\[
	\rho _{x }(A)=\rho _{x }((A^{\frac{1}{l}})^l)=\rho _{x }(A^{\frac{1}{l}})^l.
	\]
	Thus $\rho _{x }(A^{\frac{1}{l}})=\rho _{x }(A)^{\frac{1}{l}}$. If $\alpha=\frac{m}{l}$, then
	\[
	\rho _{x }(A^{\frac{m}{l}})=\rho _{x }((A^{\frac{1}{l}})^m)=\rho _{x }(A)^{\frac{m}{l}}.
	\]
	If $\alpha$ is an irrational number, then there is an increasing sequence of positive rational numbers $\alpha_k$ that converges to $\alpha$. By \cite[Theorem V.1.9]{Bhatia}, $A^{\alpha_k}$ is an increasing sequence of nonnegative matrices and so Proposition \ref{MPORP} 
	implies
	\[
	\rho _{x }(A^{\alpha})=\rho _{x }(\lim _{k\to \infty }A^{\alpha_k})=\lim _{k\to \infty }\rho _{x }(A^{\alpha_k})
	=\lim _{k\to \infty }\rho _{x }(A)^{\alpha_k}=\rho _{x }(A)^{\alpha}.
	\]
\end{proof}

%
%
%
\begin{rem} {\rm  By (\ref{local_max}) and (\ref{max_rho_x}) Theorems \ref{Schurt}, \ref{sup_inf_t_ji}, \ref{max_potence} and \ref{MTHM1}, Proposition \ref{radius1}, Lemma \ref{lemma_rhot} and Remark \ref{low_c}   also hold when $e_i$ is replaced by any $x\in \mathbb{R}^n _+ .$}
\end{rem}

We conclude this section by recalling a spectral mapping theorem for power series in max algebra and point out two versions of spectral mapping theorem for distinguished spectrum.
Let \[\mathcal{A}_+=\{f=\sum_{j=0}^\infty\alpha_jz^j: \alpha_j\ge 0, j=0,1,\dots \}.\]
For $f\in\mathcal{A}_+$, $f=\sum_{j=0}^\infty\alpha_j z^j$ write $R_f=\liminf_{j\to\infty}\alpha_j^{-1/j}$ and 
	for $0\le t<R_f$ write $f_\otimes(t)=\sup\{\alpha_jt^j:j=0,1,\dots\}$
	(note that for $t<R_f$ we have $\sup_j\alpha_jt^j<\infty$).

	Let $A\in \zR_+ ^{n\times n}$, $f=\sum_{j=0}^\infty \alpha_jz^j\in\mathcal{A}_+$, $r_{\otimes}(A)<R_f$. Define
	
	\[f_\otimes(A)=\bigoplus_{j=0}^\infty \alpha_jA^j _{\otimes}.\]
	Since $r_{\otimes}(A)=\lim_{j\to\infty}\|A^j _{\otimes}\|^{1/j}$, this definition makes sense and $f\mapsto f_\otimes(A)$ defines an analytic functional calculus for $A\in  \zR_+ ^{n\times n}$ with properties analogous to the polynomial  functional calculus (see \cite{MP15}). 
	
	The following spectral mapping theorem for max power series was established in \cite[Theorem 3.8]{MP15}.
	
	\begin{thm}\label{power}
		Let $f\in\mathcal{A}_+$ and $A\in  \zR_+ ^{n\times n}$ such that $r_{\otimes}(A)<R_f$. Then
		$r_{e_j}(f_\otimes(A))=f_\otimes(r_{e_j}(A))$ for all  $j\in \{1, \ldots, n\}$ and so $\sigma _{\otimes}(f_\otimes(A))=f_\otimes (\sigma _{\otimes}(A))$. 
		In particular, $ r_{\otimes}(f_\otimes(A))=f_\otimes (r_{\otimes}(A))$. 
	\end{thm}
	%
	By Theorem \ref{power} and (\ref{local_max})
	we  have the following result.
	
	\begin{cor}
		Let  $x\in  \zR_+ ^{n}$ be a non zero vector. If $f\in\mathcal{A}_+$ and $A\in  \zR_+ ^{n\times n}$ such that $r_{\otimes}(A)<R_f$, then $r_{x}(f_\otimes(A))=f_\otimes(r_{x}(A))$.
	\end{cor}

Next we prove a version of the spectral mapping theorem for power series 
for the distinguished part of the spectrum. 
It follows from \cite[Theorem 3.5 and remarks in Section 7]{Schneider}. Its proof is similar to the proof of \cite[Theorem 3.6]{Schneider}.

\begin{thm} Let $f=\sum_{j=0}^\infty\alpha_jz^j$ be real power series and $A\in  \zR_+ ^{n\times n}$ such that $\rho(A)<R_f=\liminf_{j\to\infty}|\alpha_j |^{-1/j}$ and $f(A)\in  \zR_+ ^{n\times n}$. 
			Then $\sigma _D(f(A))=f (\sigma _D(A))$.
			\label{gen_matr_sp_map}			
\end{thm}
\begin{proof} First we prove that $f (\sigma _D(A)) \subset \sigma _D(f(A))$. If $\lambda \in \sigma _D(A)$, there exists  $x\in \zR_+ ^{n}$, $x\neq 0$ such that $Ax=\lambda x$. It follows that $f(A)x= \sum_{j=0}^\infty\alpha_j A^j x = \sum_{j=0}^\infty\alpha_j \lambda ^j x = f(\lambda) x$ and so $f(\lambda) \in \sigma _D(f(A))$.

To prove the reverse inclusion let $\lambda \in \sigma _D(f(A))$. Since $A$ and $f(A)$ commute it follows by  \cite[Theorem 3.5]{Schneider} that there exist $x\in \zR_+ ^{n}$, $x\neq 0$ and $\beta \in \sigma _D(A)$ such that $f(A)x=\lambda x$ and $Ax=\beta x$. Since $x\neq 0$ and
$$\lambda x= f(A)x= \sum_{j=0}^\infty\alpha_j \beta ^j x=f(\beta) x$$
it follows that $\lambda = f(\beta) \in f (\sigma _D(A)) $, which completes the proof.
\end{proof}

	\begin{rem}{\rm Theorem \ref{gen_matr_sp_map} can be applied e.g. to $\exp (A)$, $\sinh(A)$, $\cosh(A)$ and to resolvent $R(\lambda, A)$ for $\lambda > \rho(A)$.
		}
	\end{rem} 
	\begin{ex}  Let $A=\left[\begin{array}{cc} 0 & \frac{1}{3} \\\frac{1}{3} & 0 \\ 
	\end{array}\right]$. Then $A^{2n}=\left[\begin{array}{cc}  \frac{1}{3^{2n}} &0 \\0 & \frac{1}{3^{2n}} \\ 
	\end{array}\right]$ and \\
	$A^{2n+1}=\left[\begin{array}{cc} 0 & \frac{1}{3^{2n+1}} \\\frac{1}{3^{2n+1}} & 0 \\ 
	\end{array}\right]$ for $n=1,2,\cdots$. Therefore $\|A^ne_1\|=\|A^ne_2\|=\frac{1}{3^n}$ for all $n\ge 1$ which implies $\rho_{e_1}(A)=\rho_{e_2}(A)=\frac{1}{3}$. By Lemma \ref{max_rho_x}
$\rho_{x}(A)=\frac{1}{3}$  for all  $x\in  \zR_+ ^{n}$, $x\neq 0$.

Since $exp(x)=\sum_{n=0}^{\infty}\frac{t^n}{n!}$, $\cosh(x)=\sum_{n=0}^{\infty}\frac{t^{2n}}{(2n)!}$ and $\sinh(t)=\sum_{n=0}^{\infty}\frac{t^{2n+1}}{(2n+1)!}$, Theorem \ref{gen_matr_sp_map} implies that 
$\rho_{x}(exp(A))=exp(\frac{1}{3}), \rho_{x}(\cosh(A))=\cosh(\frac{1}{3})$, and $\rho_{x}(\sinh(A))=\sinh(\frac{1}{3})$.
In fact  for all $f\in\mathcal{A}_+$ such that $\frac{1}{3}<R_f$, we have $\rho_{x}(f(A))=f(\frac{1}{3})$ for all  $x\in  \zR_+ ^{n}$, $x\neq 0$. Observe that 
 since $A$ is a self adjoint matrix, there is a unitary matrix $U$ such that $A=U^*DU$, where $D=\left[\begin{array}{cc} \frac{1}{3} & 0 \\0 & \frac{-1}{3} \\ 
	\end{array}\right]$ and so $f(A)=U^*\left[\begin{array}{cc} f(\frac{1}{3}) & 0 \\0 & f(\frac{-1}{3}) \\ 
	\end{array}\right]U$.
\end{ex}

\section{New inequalities on Hadamard products}

In this section we prove new inequalities on eigenvalues in max algebra and on distinguished classical eigenvalues for Hadamard products and Hadamard weighted geometric means of nonnegative matrices.

The following result was obtained in \cite[Corollary 4.8]{P12} and \cite[Theorem 5.4]{RLP19} (see also \cite{MP18}).

\begin{thm}\label{power1}
	Let $A_1, \ldots, A_m, A, B$ be $n \times n$ non-negative matrices and  let \\
	$P_i =A_i \otimes  A_{i+1} \otimes \cdots  \otimes A_m \otimes A_1 \otimes \cdots \otimes A_{i-1}$ 
	for $i=1, \ldots, m$. Then the following inequalities hold:
	\begin{equation}\label{Humu}
		r_{\otimes} (A _1  \circ \cdots \circ A _m) \le  r _{\otimes} (P_1  \circ \cdots \circ P_m)^{1/m}  \le r_{\otimes}(A_1 \otimes  \cdots \otimes A_m),
	\end{equation}
	\begin{equation}\label{maxmixmax}
		r_{\otimes} (A \circ B) \le r_{\otimes} ((A \otimes B) \circ (B \otimes A))^{1/2}\le  r_{\otimes} (A \otimes B),
	\end{equation}
	\begin{equation}\label{kvmax}
		r_{\otimes} (A\circ B) \le \|A\circ B\|\le r_{\otimes} (A^T \otimes B),
	\end{equation}
	\begin{equation}\label{Hadamard}
		r_{\otimes} (A \otimes B \circ B \otimes A) \le r_{\otimes} (A ^2 _{\otimes} \otimes B ^2 _{\otimes}).
	\end{equation}
\end{thm}

\begin{rem}{\rm
		\label{radius}
		If $r_{\otimes} (A \otimes B) \le a_{ii}b_{ii}$ for some $i\in \{1,\dots,n\}$ then $r_{\otimes} (A \otimes B)=r_{\otimes} (A\circ B)=a_{ii}b_{ii}.$

		Indeed, by applying (\ref{maxmixmax}) we have 
		$$ a_{ii}b_{ii}\leq r_{\otimes} (A\circ B)\leq r_{\otimes} (A \otimes B) \le a_{ii}b_{ii}$$
		and this proves the statement.}
\end{rem}
Since $r_{\otimes}(A) = \max _i r_{e_i} (A)$, 
it is reasonable to ask whether the analogue of Theorem \ref{power1} holds for each $r_{e_i}(A)$. It turns out that this is not the case as we will see below.

We will need the following result, which is already known in the case when $\alpha _j >0$ for $j=1, \ldots ,m$ and $s_m = \sum _{j= 1} ^m \alpha _j = 1$  \cite[Theorem 3.4]{P06}.
Note however, that Proposition \ref{radius2} is not a special case of  \cite[Theorem 4.3]{P06}  (not even under the assumption $s_m \ge 1$), since $\|E_{kj} \otimes e_i\|$ equals $1$ when $i=j$ and it equals $0$ if $i \neq j$ 
(here $E_{kj}$ denotes the matrix with $1$ at the $kj$-th entry and $0$ elsewhere) and so the assumptions of  \cite[Theorem 4.3]{P06} are not satisfied.

\begin{prop}\label{radius2} Let  $A_1, \cdots, A_m \in  \zR_+^{n\times n}$ and let $\alpha _j >0$ for $j=1, \ldots ,m$. %
	Then
	\[r_{e_i} (A_1 ^{(\alpha _1)} \circ \cdots \circ A_m ^{(\alpha _m)}) \le r_{e_i} (A_1)^{\alpha _1}   \cdots r_{e_i} (A_m)^{\alpha _m} \]
	for all $i=1, \ldots , n$.
\end{prop}

\begin{proof} {\it Case 1.} The case $s_m =\sum _{j = 1} ^m \alpha _j   =1$ is a special case of \cite[Theorem 3.4]{P06}.
	
\noindent {\it Case 2.}	In general denote $s_m=\sum _{j = 1} ^m \alpha _j  >0$ and  define $\beta _j = \frac{\alpha _j}{s_m}$.  Since $\alpha _j = \beta _j s_m$, we have by Case 1
	\begin{equation}
		\begin{aligned}
			r_{e_i} (A_1 ^{(\alpha _1)} \circ \cdots \circ A_m ^{(\alpha _m)})
			&= r_{e_i} ((A_1 ^{(\beta _1)} \circ \cdots \circ A_m ^{(\beta _m)})^{(s_m)}) \\
			&= r_{e_i} (A_1 ^{(\beta _1)} \circ \cdots \circ A_m ^{(\beta _m)})^{s_m} \\
			&\le   (r_{e_i} (A_1)^{\beta _1 }   \cdots r_{e_i} (A_m)^{\beta _m })^{s_m} \;\; (\mathrm{since} \;\; \sum _{j= 1} ^m \beta _j = 1)\\
			&= r_{e_i} (A_1)^{\alpha _1}   \cdots r_{e_i} (A_m)^{\alpha _m}, 
		\end{aligned}
	\end{equation}
	which completes the proof.
\end{proof}

\begin{cor}\label{radius3} Let  $A_1, \cdots, A_m \in  \zR_+^{n\times n}$. Then 
	\[r_{e_i} (A_1 \circ \cdots \circ A_m ) \le r_{e_i} (A_1)  \cdots r_{e_i} (A_m)  \]
	for all $i=1, \ldots , n$.
\end{cor}
It is well known that $r_{\otimes} (A\otimes B) = r_{\otimes} (B\otimes A)  $. However, the 
following example shows that the equality $ r_{e_i} (A\otimes B) = r_{e_i} (B\otimes A)$ (and also the inequality $ r_{e_i} (A\otimes B)\le  r_{e_i} (A)  r_{e_i} (B)  $)  is not correct in general.

\begin{ex}\label{exam8}
	Let $A=\left[\begin{array}{cc} 1 & 0 \\0 & 0 \\ 
	\end{array}\right]$ and $B=\left[\begin{array}{cc} 1 & 2 \\3 & 4 \\ 
	\end{array}\right]$. Then $r_{e_1} (A) =1$,  $r_{e_2} (A) =0$ and
	$A\otimes B=\left[\begin{array}{cc} 1 & 2 \\0 & 0 \\ 
	\end{array}\right]$ and $B\otimes A=\left[\begin{array}{cc} 1 & 0 \\3 & 0 \\ 
	\end{array}\right]$. So $r_{e_1} (A\otimes B)= r_{e_2} (A\otimes B) = r_{e_1} (B\otimes A)=1$ and $r_{e_2} (B\otimes A)=0$. Therefore 
	$r_{e_2} (A\otimes B)>  r_{e_2} (B\otimes A)$ and $r_{e_2} (A\otimes B) > r_{e_2} (A) r_{e_2} (B)$.
\end{ex}


\begin{rem} {\rm
		The analogue of Remark \ref{radius} does not hold for arbitrary $r_{e_j}$. More precisely,
		it may happen that $r_{e_j}(A\otimes B)=a_{ii}b_{ii} $ for some $i\in \{1,\dots,n\}$ but
		$$r_{e_j}(A\otimes B)\neq r_{e_j}(A\circ B). $$
		For instance, if $A$ and $B$ are matrices  from Example \ref{exam8} we have 
		$r_{e_2} (A\otimes B)=1=a_{11}b_{11}$, but $A\circ B=
		A=\left[\begin{array}{cc} 1 & 0 \\0 & 0 \\ 
		\end{array}\right]$ 
		and $ r_{e_2} (A\circ B)=0 $. }
\end{rem}


Let $A\in \zR_+^{n\times n} $  and $k \in \mathbb{N}$. It is known that $r_{e_i} (A^k _{\otimes}) = r_{e_i} (A)^k $  for all $i=1, \ldots , n$ (see e.g. Theorem \ref{power} below).  It is also known that 
\begin{equation}\label{ineq}
	(A_{11} \circ \cdots \circ A_{1m}) \otimes \cdots \otimes (A_{k1} \circ \cdots  \circ  A_{km})
	\le (A_{11} \otimes \cdots \otimes  A_{k1}) \circ \cdots \circ  (A_{1m} \otimes \cdots \otimes  A_{km})
\end{equation}
for $A_{jl} \in \zR_+^{n\times n}$, $j=1, \ldots, k$, $l=1, \ldots ,m$ (see e.g. \cite{P06}). A similar technique as the one used in \cite{P18b} gives the following result.

\begin{thm}\label{radius4} Let $A_1, \ldots, A_m \in \zR_+^{n\times n}$  and  let 
	$P_j =A_j  \otimes A_{j+1} \otimes  \cdots  A_m \otimes  A_1 \otimes  \cdots  \otimes A_{j-1}$ for $j=1, \ldots, m$. Then
	\[r_{e_i} (A_1 \circ \cdots \circ A_m ) \le r_{e_i} ^{1/m} (P_1 \circ \cdots \circ P_m )\le (r_{e_i} (P_1) \cdots r_{e_i} (P_m))^{1/m} \]
	for every  $1\le i \le n$.
\end{thm}

\begin{proof} We have
	\[r_{e_i} ^m (A_1 \circ \cdots \circ A_m ) = r_{e_i} ((A_1 \circ \cdots \circ A_m)^m  _{\otimes}) \]
	\[=  r_{e_i}  ((A_1 \circ \cdots \circ A_m) \otimes  (A_2 \circ \cdots \circ A_m \circ A_1) \otimes \cdots \otimes  (A_m \circ \cdots \circ A_{m-2} \circ A_{m-1}) )\]
	\[\le  r_{e_i} (P_1 \circ \cdots \circ P_m )\le r_{e_i} (P_1) \cdots r_{e_i} (P_m),\]
	which completes the proof.
\end{proof}

\begin{cor}\label{radius5}  If $A, B \in \zR_+^{n\times n}$, then for every  $1\le i \le n$
	\[r_{e_i} (A \circ B ) \le r_{e_i} ^{1/2} ((A\otimes B) \circ (B\otimes A))\le (r_{e_i} (A\otimes B)  r_{e_i} (B\otimes A))^{1/2}.\]
	
\end{cor}

The following example shows that the inequalities $r_{e_i} (A \circ B ) \le r_{e_i} (A \otimes B )$ and $r_{e_i} (A \circ B ) \le r_{e_i} (A^T \otimes B )$ are not correct in general.

\begin{ex}
	Let $A=\left[\begin{array}{cc} 0 & 1 \\1 & 0 \\ 
	\end{array}\right]$ and $B=\left[\begin{array}{cc} 0 & 1 \\\frac{1}{4} & 0 \\ 
	\end{array}\right]$. Then $A\circ B =B$, $A^T =A$ and
	$A\otimes B=\left[\begin{array}{cc} \frac{1}{4} & 0 \\0 & 1 \\ 
	\end{array}\right]$. So 
	\[r_{e_1} (A\circ B)= \frac{1}{2} >  \frac{1}{4} =  r_{e_1} (A \otimes B ) =  r_{e_1} (A^T \otimes B ).\]
\end{ex}

The above results (Propositions \ref{radius1} and \ref{radius2}, Theorem \ref{radius4}, Corollaries \ref{radius3} and \ref{radius5}) hold in fact for all $x \in \zR ^n _+$ not just for all $e_i$. It is trivial that they are valid for $x=0$. Moreover since for $x\neq 0$  the equality (\ref{local_max}) is valid,  
the following result holds. 

\begin{thm}\label{Lradius} Let $x \in \zR ^n _+$,  $A_1, \cdots, A_m, A, B \in  \zR_+^{n\times n}$, $t>0$ and let $\alpha _j$ for $j=1, \ldots ,m$ be positive numbers. 
	Let 
	$P_j =A_j  \otimes A_{j+1} \otimes \cdots  A_m \otimes  A_1 \otimes  \cdots \otimes A_{j-1}$ for $j=1, \ldots, m$. Then
	\begin{enumerate}
		\item $r_{x} (A^{(t)}) = r_{x} (A)^t,$
		\medskip
		
		\item $r_{x} (A_1 ^{(\alpha _1)} \circ \cdots \circ A_m ^{(\alpha _m)}) \le r_{x} (A_1)^{\alpha _1}   \cdots r_{x} (A_m)^{\alpha _m}, $
		\medskip
		
		\item $r_{x} (A_1 \circ \cdots \circ A_m ) \le r_{x} (A_1)  \cdots r_{x} (A_m),$
		\medskip
		
		\item $r_{x} (A_1 \circ \cdots \circ A_m ) \le r_{x} ^{1/m} (P_1 \circ \cdots \circ P_m )\le (r_{x} (P_1) \cdots r_{x} (P_m))^{1/m}, $
		\medskip
		
		\item $r_{x} (A \circ B ) \le r_{x} ^{1/2} ((A\otimes B) \circ (B\otimes A))\le (r_{x} (A\otimes B)  r_{x} (B\otimes A))^{1/2}.$
	\end{enumerate}
\end{thm}

	\bigskip
	
	Next we  formulate suitable analogues of the above result for distinguished eigenvalues  of  non-negative matrices.
	
	The following result for Hadamard weighted geometric means  was proved in \cite{P06}.
	\begin{thm}
		\label{good_work}
		Let 
		$\{A_{i j}\}_{i=1, j=1}^{l, m}$ be
		nonnegative $n\times n$ matrices, $x \in \zR ^n _+$  
		and
		$\alpha _1$, $\alpha _2$,..., $\alpha _m$ nonnegative numbers such that
		$s_m= \sum_{i=1}^m \alpha _i \ge 1$.
		Let 
		$$ \;\;\;\; \;\; \;\;\;B:= \left(A_{1 1}^{(\alpha _1)} \circ \cdots \circ A_{1 m}^{(\alpha _m)}\right) \ldots \left(A_{l 1}^{(\alpha _1)} \circ \cdots \circ A_{l m}^{(\alpha _m)} \right).$$ Then we have
		\begin{equation}
			\label{basic2}
			B \le  
			(A_{1 1} \cdots  A_{l 1})^{(\alpha _1)} \circ \cdots 
			\circ (A_{1 n} \cdots A_{l n})^{(\alpha _n)} . \\
		\end{equation}
		If, in addition, $s_m =1$, then
		\begin{equation}
		\nonumber
			\rho _x (A_1 ^{( \alpha _1)} \circ A_2 ^{(\alpha _2)} \circ \cdots \circ A_m ^{(\alpha _n)} ) \le
			\rho _x (A_1)^{ \alpha _1} \, \rho _x (A_2)^{\alpha _2} \cdots \rho _x (A_m)^{\alpha _m} ,
		\end{equation}
		\begin{eqnarray}
		\nonumber
			\rho _x \left(B \right)   &\le &  
			\nonumber
			\rho _x \left( (A_{1 1} \cdots  A_{l 1})^{(\alpha _1)} \circ \cdots 
			\circ (A_{1 m} \cdots A_{l m})^{(\alpha _m)} \right) \\
			&\le  &
			\rho _x \left( A_{1 1} \cdots  A_{l 1} \right)^{\alpha _1} \cdots 
			\rho _x \left( A_{1 m} \cdots A_{l m}\right)^{\alpha _m} .
			\nonumber
		\end{eqnarray}
	\end{thm}
	\bigskip
	The following result improves Theorem \ref{good_work}.
	\begin{thm}
	\label{good_theorem}
	 Let $A_1, \ldots , A_m$ and $\{A_{i j}\}_{i=1, j=1}^{l, m}$ be
		nonnegative $n\times n$ matrices, $x \in \zR ^n _+$,  $t \ge 1$ and
		$\alpha _1$, $\alpha _2$,..., $\alpha _m$ nonnegative numbers such that
		$s_m= \sum_{i=1}^m \alpha _i \ge 1$. Let $B$ be defined as in Theorem \ref{good_work}.
		Then we have
		\begin{equation}
			\label{rhoxt}
			\rho _x (A^{(t)}) \le \rho _x (A)^t,
		\end{equation}
		\begin{equation}
			\label{t2}
			\rho _x (A_{1} ^{(t)} \cdots  A_{m} ^{(t)}) \le \rho _x (( A_1 \cdots A_m )^{(t)}) \le  \rho _x (A_1 \cdots A_m )^t,
		\end{equation}
		\begin{equation}
			\rho _x (A_1 ^{( \alpha _1)} \circ A_2 ^{(\alpha _2)} \circ \cdots \circ A_m ^{(\alpha _n)} ) \le
			\rho _x (A_1)^{ \alpha _1} \, \rho _x (A_2)^{\alpha _2} \cdots \rho _x (A_m)^{\alpha _m} ,
			\label{gl1vecr}
		\end{equation}
		\begin{eqnarray}
		\nonumber
			\rho _x \left(B \right)   &\le &  
			\nonumber
			\rho _x \left( (A_{1 1} \cdots  A_{l 1})^{(\alpha _1)} \circ \cdots 
			\circ (A_{1 m} \cdots A_{l m})^{(\alpha _m)} \right) \\
			&\le  &
			\rho _x \left( A_{1 1} \cdots  A_{l 1} \right)^{\alpha _1} \cdots 
			\rho _x \left( A_{1 m} \cdots A_{l m}\right)^{\alpha _m} .
			\label{very_good}
		\end{eqnarray}
		
	\end{thm}
	\begin{proof} We may assume that $x\neq 0$. By (\ref{max_rho_x}) and  Lemma \ref{lemma_rhot}  there exists $i$ such that
	$$\rho _x (A^{(t)})= \rho _{e_i} (A^{(t)})  \le \rho _{e_i} (A)^t \le \rho _x (A)^t,$$
	which proves (\ref{rhoxt}).
		Inequalities (\ref{t2}) follow from (\ref{t_dobro}) and (\ref{rhoxt}).
		
		To establish  (\ref{gl1vecr}) let $\beta _i = \frac{\alpha _i}{s_m}$, $i=1, \ldots , m$. Then $\sum_{i=1}^m \beta _i = 1$ and so we have by (\ref{rhoxt}) and Theorem \ref{good_work}
		$$\rho _x (A_1 ^{( \alpha _1)} \circ A_2 ^{(\alpha _2)} \circ \cdots \circ A_m ^{(\alpha _n)} ) \le \rho _x (A_1 ^{( \beta _1)} \circ A_2 ^{(\beta _2)} \circ \cdots \circ A_m ^{(\beta _n)} )^{s_m}  $$
		$$\le (\rho _x (A_1)^{ \beta _1} \, \rho _x (A_2)^{\beta _2} \cdots \rho _x (A_m)^{\beta _m})^{s_m} = \rho _x (A_1)^{ \alpha _1} \, \rho _x (A_2)^{\alpha _2} \cdots \rho _x (A_m)^{\alpha _m},$$
		which proves (\ref{gl1vecr}). Inequalities (\ref{very_good}) follow from (\ref{basic2}) and 
		(\ref{gl1vecr}).

	\end{proof}
	
	\begin{cor}
		\label{first}
		Suppose that $x\in \zR _+ ^{n}$ and $A_1,A_2,\dots,A_m,A, B \in \zR _+ ^{n\times n}$. 
		Then

		$$\rho _{x}(A_1\circ\dots \circ A_m)\leq \rho_x(A_1)\dots \rho_x(A_m),$$
		$$\rho_x(A\circ B)\le \rho_x(A)\rho_x(B),$$
	\end{cor}
	The following result is proved a similar manner as Theorem \ref{radius4}.
	\begin{prop} 
		\label{some_P}
		Let $A, B, A_1, \ldots , A_m$  be
		nonnegative $n\times n$ matrices, $x\in \zR _+ ^{n}$ and let  $P_j=A_j A_{j+1}~\dots~A_mA_1~\dots~A_{j-1}$ for $j=1,\dots,m$.
		
		Then 
		\begin{equation}
			\label{analogue_P_m}
			\rho _{x}(A_1\circ\dots \circ A_m)\leq \rho _{x}(P_1\circ\dots \circ P_m)^{1/m}\leq (\rho _x (P_1)\dots \rho_x(P_m))^{1/m}.
		\end{equation}
		In particular,
		\begin{equation}
			\label{P_2}
			\rho _{x}(A\circ B)\leq {\rho_x}(AB \circ BA)^{\frac{1}{2}}\leq {\rho _{x}(AB)^{\frac{1}{2}} \rho _{x}(BA)}^{\frac{1}{2}}.
		\end{equation}
	\end{prop}

	\begin{rem}{\rm Let $A, B, A_1, \ldots , A_m,P_1, \ldots , P_m $ be as in Proposition \ref{some_P}. If,  in addition,  
 $\alpha_1,\dots ,\alpha_m>0$ and $A_1,A_2,\dots,A_m$ are positive semidefinite matrices 
			then
			\[
			\rho_{x}({A_1}^{\alpha_1}\circ\cdots \circ A_m^{\alpha_m})\leq \rho_x(A_1)^{\alpha_1}\cdots \rho_x(A_m)^{\alpha_m}.\]
		}
	\end{rem}

	\bigskip
	
	We conclude the article by pointing out a norm inequality on a triple Jordan product in max algebra (Proposition \ref{norm1}).
	If $A = [a_{ij}]$ is a non-negative matrix, let $A^T = [a_{ji}]$ denote  a transpose matrix of $A$. 
	The following proposition is known (\cite{RLP19})  and easy to verify.
	
	\begin{prop}\label{prop3.1}
		Let $A\in  \zR_+ ^{n\times n}$. Then
		\[\|A\|^2= r_{\otimes} (A^T \otimes A) = r_{\otimes} (A\otimes A^T).\]
	\end{prop}
	The following result was established in  \cite[Theorem 5.4]{RLP19}.
	\begin{thm} Let $A,B, A_1, \ldots, A_m  \in \zR_+ ^{n\times n}$. 

	\noindent If $m$ is even, then
	\begin{equation}
	\label{th5sod_notref}
	 \|A_1 \circ A_2 \circ \cdots \circ A_m \|^2 \le 
	 \end{equation}
	 $$
	 r_{\otimes} (A_1 ^T \otimes A_2 \otimes A_3 ^T \otimes A_4 \otimes \cdots A_{m-1} ^T \otimes A_m)\,  r_{\otimes} (A_1 \otimes A_2 ^T \otimes A_3 \otimes A_4 ^T \otimes \cdots A_{m-1} \otimes A_m ^T  ) $$
	\begin{equation}
	\nonumber
	=  r_{\otimes} (A_1 ^T \otimes A_2 \otimes A_3 ^T \otimes A_4 \otimes  \cdots \otimes  A_{m-1} ^T \otimes  A_m) \, r_{\otimes} (A_m \otimes A_{m-1} ^T \otimes \cdots A_{4} \otimes A_3 ^T \otimes  A_2\otimes A_1 ^T ).
	\end{equation}

	\noindent If $m$ is odd, then
	\begin{equation}
	\label{th5lih_notref}
	\|A_1 \circ A_2 \circ \cdots \circ A_m \|^2 \le 
	\end{equation}  
	\begin{equation}
	\nonumber
	r_{\otimes} (A_1 \otimes  A_2 ^T \otimes  \cdots \otimes A_{m-2} \otimes A_{m-1} ^T \otimes  A_m \otimes A_1 ^T \otimes A_2 \otimes A_3 ^T \otimes A_4 \otimes  \cdots A_{m-2}^T \otimes  A_{m-1}\otimes A_m ^T  )
	\end{equation}
	In particular, Inequalities (\ref{kvmax}) hold.
	\label{th5H}
\end{thm}
	
	The following result  refines  (\ref{kvmax}). 
	
	\begin{prop}\label{norm}  Let $A, B \in  \zR_+^{n\times n}$. Then
		\[\|A \circ B \|\le r_{\otimes} ^{1/2} ((A^T \otimes B)\circ (B^T \otimes A) ) \le r_{\otimes} (A^T \otimes B).\]
	\end{prop}
	\begin{proof} By Proposition \ref{prop3.1} and (\ref{ineq}) we have
		\[\|A \circ B \|=r_{\otimes} ^{1/2} ((B^T\circ  A^T)\otimes (A \circ B) ) \le r_{\otimes} ^{1/2} ( (B^T \otimes A) \circ (A^T \otimes B)) \le r_{\otimes} (A^T\otimes B),\]
		since $r_{\otimes} (B^T \otimes A ) =  r_{\otimes} (A^T\otimes B). $
	\end{proof}
	
	The following result that follows from (\ref{th5lih_notref}) 
	is a max algebra version of  \cite[Theorem 5]{Huang} (see also \cite{DP16}).\\

	\begin{prop}\label{norm1}
		Let $A,B$ be non-negative $n \times n$ matrices. Then 
		\[\Vert A \circ B^T\circ A \Vert \leq {\Vert A\otimes B \otimes A\Vert} \]
	\end{prop}
	
	\begin{proof} %
		It follows from \cite[Theorem 5.4]{RLP19} and  Proposition \ref{prop3.1} that
		$$\Vert A \circ B^T\circ A \Vert  \le  r_{\otimes}  ^{\frac{1}{2}} (A  \otimes  B  \otimes  A  \otimes  A ^T   \otimes B^T  \otimes  A^T)  
		={\Vert A\otimes B \otimes A\Vert}, $$
		which completes the proof.
	\end{proof}

\subsection*{Acknowledgement} 

The authors thank the reviewer for his comments, suggestions and careful reading that improved the article considerably.
	
\subsection*{Funding} 	
This first author and the third author were supported  by the Department of Mathematical Sciences at
Isfahan University of Technology, Iran.  The second author acknowledges a partial support of the Slovenian Research Agency (grants P1-0222, J1-8133, J1-8155, J2-2512 and N1-0071).
%


\begin{thebibliography}{1}

\bibitem{AGB01} M. Akian, R.B. Bapat, S. Gaubert,  \textit{Generic asymptotics of eigenvalues
using Min-Plus algebra}. Proceedings of the Workshop on Max-Plus Algebras. IFAC
SSSC’01, Elsevier. 2001. 

\bibitem{AGW}  M. Akian, S. Gaubert, C. Walsh,  
\textit{Discrete max-plus spectral theory}. Idempotent Mathematics and Mathematical Physics, G.L. Litvinov and V.P. Maslov, Eds, Contemporary Mathematics. 2005, 377, 53--77, AMS. 


\bibitem{Audenaert} K.M.R. Audenaert, \textit{Spectral radius of Hadamard product versus conventional product for
non-negative matrices}. Linear Algebra Appl. 2010, 432, no. 2, 366–-368.

\bibitem{BCOQ92}  F.L. Baccelli, G. Cohen, G.J. Olsder, J.-P. Quadrat, \textit{Synchronization
and linearity}, Wiley Series in Probability and Mathematical Statistics. Probability and Mathematical
Statistics, John Wiley and Sons, Ltd., Chichester. 1992. 

\bibitem{Bapat} R.B. Bapat,  
\textit{A max version of the Perron-Frobenius theorem.}
Linear Algebra Appl. 1998, 275-276, 3--18. 

\bibitem{Bhatia}R. Bhatia, \textit{Matrix Analysis}. Springer. 1998.

\bibitem{Butkovic} P. Butkovi\v{c},
\textit{Max-linear systems: theory and algorithms.} 
Springer-Verlag, London. 2001.

\bibitem{BGC-G09} P. Butkovi\v{c}, S. Gaubert, R. A. Cuninghame-Green, \textit{Reducible spectral theory with applications to the robustness of matrices in max-algebra, SIAM J. Matrix Anal. Appl.} 2009, 31(3), 1412--1431.

\bibitem{BSST13} 
P. Butkovi\v{c}, H. Schneider, S. Sergeev, B.S. Tam, 
\textit{Two cores of a nonnegative matrix.} Linear Algebra Appl. 2013, 439, 1929--1954.

\bibitem{CZ} D. Chen, Y. Zhang, \textit{On the spectral radius of Hadamard products of nonnegative matrices.} Banach J. Math. Anal.  9, no. 2015, 2, 127--133. 

\bibitem{Cuninghame} R.A. Cuninghame-Green, \textit{ Minimax algebra.} In Lecture Notes in Economics and Mathematical Systems. vol. 166, Springer, Berlin. 1979.

\bibitem{DP16}  R. Drnov\v{s}ek, A. Peperko, \textit{Inequalities on the spectral radius and the operator norm of Hadamard products of positive operators on sequence spaces.} Banach J. Math. Anal. 10, no. 2016, 4, 800--814. 

\bibitem{Elsner}  L. Elsner,  C.R. Johnson,  J.A. Dias Da Silva, 
\textit{The Perron root of a weighted geometric mean of nonnegative matrices.}
Linear Multilinear Algebra. 1988, 24, 1--13.    


\bibitem{GOW19} N. Guglielmi,
O. Mason, F. Wirth, \textit{Barabanov norms, Lipschitz continuity and monotonicity for
the max algebraic joint spectral radius.}  Linear Algebra Appl. 2018, 550, 37--58.

\bibitem{GPZ20} B. Gabrov\v{s}ek, A. Peperko, \v{Z}erovnik. \textit{Independent Rainbow Domination Numbers of
Generalized Petersen Graphs $P(n, 2)$ and $P(n, 3)$.} Mathematics 2020, 8(6), 996, 1--13. 

\bibitem{Gunawardena}J. Gunawardena, \textit{Cycle times and fixed points of min-max functions.} In G. Cohen and J.-P. Quadrat, editors. 1994, 11th
International Conference on Analysis and Optimization of Systems,   Springer LNCIS
199,  266--272. 

\bibitem{atwork} B. Heidergott, G.J Olsder, J. van der Woude, \textit{Max plus at work.} Princeton University Press. 2006.

\bibitem{HoJo1}  R.A Horn, C.R Johnson, \textit{Topics in Matrix Analysis.} Cambridge University Press. 1991.

\bibitem{HZ10} R.A. Horn, F. Zhang, \textit{Bounds on the spectral radius of a Hadamard product of nonnegative or positive semidefinite matrices.} 
Electron. J. Linear Algebra. 2010, 20, 90--94.

\bibitem{Huang} Z. Huang, \textit{On the spectral radius and the spectral norm of Hadamard products of nonnegative
matrices.} Linear Algebra Appl. 2011, 434, 457--462.


\bibitem{Gaubert} S. Gaubert, \textit{Théorie des systemes linéaires dans les dioïdes.} Thèse, École des Mines des Paris. 1992.




\bibitem{Schneider} R.D. Katz, H. Schneider, S. Sergeev,  \textit{On commuting matrices in max algebra and in nonnegative matrix algebra.} Linear Algebra Appl. 2012, 436(2), 276--292.

\bibitem{Litvinov07} G.L. Litvinov, 
\textit{The Maslov dequantization, idempotent and tropical mathematics.} A brief introduction, 
J. Math. Sci.(N. Y.). 2007,  140, no.3, 426--444.

\bibitem{LM05} G.L. Litvinov, V.P. Maslov,
\textit{Idempotent mathematics and mathematical physics.} 
Contemp. Math. 377, Amer.Math. Soc., 2005.



\bibitem{L06}Y.Y. Lur, \textit{A max version of the generalized spectral radius theorem.} Linear Algebra  
Appl. 2006, 418, 336--346.


\bibitem{MN02} J. Mallet-Paret, R.D Nussbaum,
\textit{Eigenvalues for a class of homogeneous cone maps arising from max-plus 
operators.} 
Discrete and Continuous Dynamical Systems, 2002, vol. 8, no 3, 519--562. 


\bibitem{MP15}  V. M\"uller,  A. Peperko, 
On the spectrum in max-algebra. 
Linear Algebra Appl. 2015, 485,  250--266. 

\bibitem{MP17} V. M\"uller, A. Peperko, \textit{On the Bonsall cone spectral radius and the approximate point spectrum.} 
Discrete and Continuous Dynamical Systems - Series A,  2017, vol. 37, no 10, 5337--5364.

\bibitem{MP18} V. M\"uller, A. Peperko,  \textit{Lower spectral radius and spectral mapping theorem for supreme preserving mappings.}  Discrete and Continuous Dynamical Systems - Series A, 2018, vol. 38, no 8,  4117--4132.

\bibitem{PS05}L. Pachter,  B. Sturmfels,  \textit{Algebraic statistics for computational biology.} Cambridge Univ. Press, 2005.

\bibitem{P08} A. Peperko, \textit{On the max version of the generalized spectral radius theorem.} Linear Algebra  
Appl. 2008, 428, 2312--2318.

\bibitem{P06}  A. Peperko, \textit{Inequalities for the spectral radius of non-negative functions.}
Positivity. 2009, 13, 255--272.  

\bibitem{P11} A. Peperko,  \textit{On the continuity of the generalized spectral radius in max algebra.} Linear Algebra  
Appl. 2011, 435, 902--907.

\bibitem{P12} A. Peperko,  \textit{Bounds on the generalized and the joint spectral radius of Hadamard products of bounded sets of positive operators on sequence 
spaces.} Linear Algebra  Appl. 2012, 437, 189--201.

\bibitem{P17} A. Peperko, \textit{Bounds on the joint and generalized spectral radius of the Hadamard geometric mean of bounded sets of positive kernel operators.}  
Linear Algebra  Appl. 2017, 533, 418--427.

\bibitem{P18a} A.  Peperko,  \textit{Inequalities on the spectral radius. operator norm and numerical radius of the Hadamard products of positive kernel operators.}  Lin. Mult. Algebra. 2019, 67:8, 1637--1652.

\bibitem{P18b} A. Peperko,  \textit{Inequalities on the joint and generalized spectral and essential spectral radius of the Hadamard geometric mean of bounded sets of positive kernel operators.} Lin. Mult. Algebra. 2019, 67:11, 2159--2172.


\bibitem{RLP19} A. Rosenmann, F. Lehner, A. Peperko, \textit{Polynomial convolutions in max-plus algebra.} Linear Alg. Appl. 2019, 578, 370--401.


\bibitem{Sch10a} A.R. Schep, \textit{Bounds on the spectral radius of Hadamard products of positive operators on $l_p$-spaces.}  Electronic J. Linear Algebra. 2011, 22, 443--447.

\bibitem{Sch10b} A.R. Schep, \textit{Corrigendum for "Bounds on the spectral radius of Hadamard products of positive operators on $l_p$-spaces".} 2011, preprint, at Research gate: Corrigendum-Hadamard 

\bibitem{Z} Y. Zhang,   \textit{Some spectral norm inequalities on Hadamard products of nonnegative matrices.} Linear Algebra Appl. 2018, 556, 162--170.



 

\end{thebibliography}

 
 
 \bigskip

\noindent
Seyed Mahmoud Manjegani \\
Department of Mathematical Sciences\\
Isfahan University of Technology\\
Isfahan, Iran 84156-83111\\
e-mail: manjgani@iut.ac.ir
\bigskip

\noindent Aljo\v sa Peperko \\
Faculty of Mechanical Engineering \\
University of Ljubljana \\
A\v{s}ker\v{c}eva 6\\
SI-1000 Ljubljana, Slovenia\\
{\it and} \\
Institute of Mathematics, Physics and Mechanics \\
Jadranska 19 \\
SI-1000 Ljubljana, Slovenia \\
e-mail:  aljosa.peperko@fs.uni-lj.si \\
$*$ Corresponding author

\bigskip

\noindent Hojr Shokooh Saljooghi \\
Department of Mathematical Sciences\\
Isfahan University of Technology\\
Isfahan, Iran 84156-83111\\ University of Technology\\
e-mail: h.shokooh@math.iut.ac.ir \\
 
\end{document}